\documentclass[12pt,draft]{article}
\usepackage{mathtools}
\usepackage{amssymb}
\usepackage{amsthm}
\usepackage{xcolor}
\usepackage{tikz}
\usepackage{fullpage}

\DeclareSymbolFont{bbold}{U}{bbold}{m}{n}
\DeclareSymbolFontAlphabet{\mathbbold}{bbold}

\newcommand{\N}{\mathbb{N}}

\newcommand{\R}{\mathbb{R}}

\renewcommand{\epsilon}{\varepsilon}



\renewcommand{\Re}{\operatorname{Re}}

\let\phi\varphi

\let\leq\leqslant

\let\geq\geqslant

\makeatletter
\def\@row#1,{#1\@ifnextchar;{\@gobble}{&\@row}}
\def\@matrix{%
    \expandafter\@row\my@arg,;%
    \@ifnextchar({\\ \get@in@paren{\@matrix}}{\after@matrix}%
    }
\def\matrixtype#1#2#3{%
    \ifmmode\def\after@matrix{\end{#2}\right#3}%
    \else\def\after@matrix{\end{#2}\right#3$}$\fi
    \left#1\begin{#2}\get@in@paren{\@matrix}%
    }
\def\@column#1,{#1\@ifnextchar;{\@gobble}{\\ \@column}}
\newcommand\vect{}
\def\svect(#1){\left(\begin{smallmatrix}\@column#1,;\end{smallmatrix}\right)}
\def\vect{\get@in@paren{\@vect}}
\def\@vect{\left(\begin{matrix}\expandafter\@column\my@arg,;\end{matrix}\right)}
\def\get@in@paren#1({\def\my@arg{}\def\my@rest{}\def\after@get{#1}\get@arg}
\let\e@a\expandafter
\def\get@arg#1){\e@a\kl@test\my@rest#1(;}
\def\kl@test#1(#2;{\e@a\def\e@a\my@arg\e@a{\my@arg#1}%
                   \ifx:#2:\let\my@exec\after@get
                   \else\let\my@exec\get@arg
                        \e@a\def\e@a\my@arg\e@a{\my@arg(}%
                        \def@rest#2;%
                   \fi\my@exec}
\def\def@rest#1(;{\def\my@rest{#1\kl@zu}}
\def\kl@zu{)}

\makeatletter
\newcommand\MyPairedDelimiter{%
  \@ifstar{\My@Paired@Delimiter{{}}}
          {\My@Paired@Delimiter{}}%
}
\newcommand\My@Paired@Delimiter[4]{%
  \newcommand#2{%
    \@ifstar{\start@PD{#1}{\delimitershortfall=-1sp}{#3}{#4}}
            {\start@PD{#1}{}{#3}{#4}}%
  }%
}
\newcommand\start@PD[5]{%
  #1\mathopen{\mathpalette\put@delim@helper{\put@delim{#2}{#3}{.}{#5}}}%
  #5%
  \mathclose{\mathpalette\put@delim@helper{\put@delim{#2}{.}{#4}{#5}}}%
}
\newcommand\put@delim@helper[2]{%
  \hbox{$\m@th\nulldelimiterspace=0pt #2#1$}%
}
\newcommand\put@delim[5]{%
  \setbox\z@\hbox{$\m@th#5{#4}$}%
  \setbox\tw@\null
  \ht\tw@\ht\z@ \dp\tw@\dp\z@
  #1#5%
  \left#2\box\tw@\right#3%
}

\makeatother
\MyPairedDelimiter*{\abs}{\lvert}{\rvert}
\MyPairedDelimiter*{\norm}{\lVert}{\rVert}
\MyPairedDelimiter{\set}{\{}{\}}

\theoremstyle{plain} 
\newtheorem{theorem}{Theorem}[section]

\theoremstyle{definition}

\newtheorem*{definition}{Definition}
\newtheorem{remark}[theorem]{Remark}

\usepackage{enumitem}

\setenumerate[1]{nolistsep} 
\setenumerate[2]{nolistsep} 

\begin{document}

\medmuskip=4mu plus 2mu minus 3mu
\thickmuskip=5mu plus 3mu minus 1mu
\belowdisplayshortskip=9pt plus 3pt minus 5pt

\title{Stabilization via Homogenization}

\author{Marcus Waurick}

\date{}

\maketitle

\begin{abstract}
 In this short note we treat a 1+1-dimensional system of changing type. On different spatial domains the system is of hyperbolic and elliptic type, that is, formally, $\partial_t^2 u_n-\partial_x^2 u_n = \partial_t f$ and $u_n-\partial_x^2 u_n= f$ on the respective spatial domains $\bigcup_{j\in \{1,\ldots,n\}} \big(\frac{j-1}{n},\frac{2j-1}{2n}\big)$ and $\bigcup_{j\in \{1,\ldots,n\}} \big(\frac{2j-1}{2n},\frac{j}{n}\big)$. We show that $(u_n)_n$ converges weakly to $u$, which solves the exponentially stable limit equation $\partial_t^2 u +2\partial_t u + u -4\partial_x^2 u = 2(f+\partial_t f)$ on $[0,1]$. If the elliptic equation is replaced by a parabolic one, the limit equation is \emph{not} exponentially stable.
\end{abstract}



\section*{Acknowledgements}

This work was carried out with financial support of the EPSRC grant EP/L018802/2: “Math-
ematical foundations of metamaterials: homogenisation, dissipation and operator theory”.
This is gratefully acknowledged. The author is indebted to the referee for useful comments.

\section{Introduction}

For $n\in \mathbb{N}$ and a given smooth $f$, we consider the following equation of mixed type:
\[
   \begin{cases}
        \partial_t^2 u_n(t,x) - \partial_x^2 u_n(t,x) = \partial_t f(t,x),&  x\in \bigcup_{j\in \{1,\ldots,n\}} \big(\frac{j-1}{n},\frac{2j-1}{2n}\big),\\
        u_n(t,x) - \partial_x^2 u_n(t,x) = f(t,x),& x\in \bigcup_{j\in \{1,\ldots,n\}} \big(\frac{2j-1}{2n},\frac{j}{n}\big),\\
        (\partial_x u_n)(t,0)=(\partial_x u_n)(t,1)=0, & 
   \end{cases}
   \quad(t\in \mathbb{R}),
\] subject to zero initial conditions and conditions of continuity at the junction points $\{(2j-1)/2n; j\in \{1,\ldots,n-1\}\}$ for $u_n$. We show that for $n\to\infty$ the sequence of solutions $(u_n)_{n\in\N}$ converges weakly in $L^2_{\textnormal{loc}}(\mathbb{R}\times [0,1])$ to $u$, which solves
\begin{equation}\label{e:hh}
   \frac{1}{2}\partial_t^2 u(t,x) + \partial_t u +  \frac{1}{2}u(t,x) - 2\partial_x^2 u(t,x)= f(t,x)+\partial_t f(t,x), \quad((t,x)\in\mathbb{R}\times (0,1))
\end{equation}
subject to $\partial_xu(t,0)=\partial_xu(t,1)=0$ for $t\in \mathbb{R}$ and zero initial conditions. Moreover, we show that the asymptotic limit admits exponentially stable solutions. Note that the stability result for the limit equation is due to the superposed effect of the hyperbolic type and the elliptic type equation: Indeed, it is remarkable that $(\partial_t^2-\partial_x^2)u=\partial_t f$ is \emph{not} exponentially stable, if considered on the whole of $[0,1]$ as underlying spatial domain. Moreover, we will show that if we replace the elliptic part, $u_n(t,x) - \partial_x^2 u_n(t,x) = f(t,x)$, by a corresponding parabolic one, that is, $\partial_tu_n(t,x) - \partial_x^2 u_n(t,x) = f(t,x)$ the limit equation reads
\begin{equation}\label{e:hp}
   \partial_t^2 u(t,x) + \partial_t u - 2\partial_x^2 u(t,x)= f(t,x)+\partial_t f(t,x), \quad((t,x)\in\mathbb{R}\times (0,1))
\end{equation}
subject to homogeneous Neumann boundary conditions. Moreover, we find that the limit equation is \emph{not} exponentially stable (in the sense of \cite[Definition 3.1]{Trostorff2013a}, see also \cite[Section 3.1]{Trostorff2015a}).

For the proof of the homogenization (i.e.~the computation of the limit equation) and stability results, we will employ the notion of evolutionary equations developed in \cite{PicPhy,Picard}. We will use results on exponential stability of \cite{Trostorff2013a} (with an improvement in \cite{Trostorff2014PAMM}) developed in this line of reasoning. The computation of the limit equation is based on \cite{Waurick2014SIAM_HomFrac,Waurick2013AA_Hom}. In the next section, we will recall the notion of evolutionary equations and the results mentioned. The third section establishes the functional analytic framework for the equations to study. Moreover, we provide the proof of the result mentioned concerning the hyperbolic-elliptic system. We address the case where the parabolic equation replaces the elliptic one in the last section.

\section{Evolutionary Equations}

In the whole section, let $\mathcal{H}$ be a Hilbert space. For $\nu\in \mathbb{R}$ we define
\[
   L_\nu^2(\mathbb{R};\mathcal{H})\coloneqq \{ f\colon \mathbb{R}\to \mathcal{H}; f\text{ measurable}, \int_\R \|f(t)\|_{\mathcal{H}}^2e^{-2\nu t} dt <\infty \}
\]
endowed with the obvious norm (and scalar product). We set
\[
   \partial_{t,\nu} \colon D(\partial_{t,\nu})\subseteq L_\nu^2(\mathbb{R};\mathcal{H})\to L_\nu^2(\mathbb{R};\mathcal{H}), f\mapsto f',
\]
where $f'$ denotes the distributional derivative and $D(\partial_{t,\nu})$ is the maximal domain in $L_\nu^2(\mathbb{R};\mathcal{H})$. Note that for all $\nu\neq 0$, we have $\partial_{t,\nu}^{-1}$ is a bounded linear operator in $L_\nu^2(\mathbb{R};\mathcal{H})$, see \cite[Corollary 2.5]{Kalauch}. Note that also $\partial_{t,\nu}^{-1} f = \int_{-\infty}^{(\cdot)} f(\tau) d\tau$ for $f\in L_{\nu}^2(\mathbb{R};\mathcal{H})$ and $\nu>0$.

For a closed, densely defined linear operator $B$ in $\mathcal{H}$, we shall denote the corresponding lifted operator to $L_\nu^2(\mathbb{R};\mathcal{H})$ by the corresponding calligraphic letter, that is,
\[
   \mathcal{B} \colon L_\nu^2(\mathbb{R};D(B))\subseteq L_\nu^2(\mathbb{R};\mathcal{H})\to L_\nu^2(\mathbb{R};\mathcal{H}), f\mapsto (t\mapsto Bf(t)).
\]
The exponentially weighted $L^2$-type spaces have been used to obtain a solution theory for abstract operator equations in space time. $L(\mathcal{H})$ denotes the space of bounded linear operators in $\mathcal{H}$.
\begin{theorem}[{{\cite[Solution Theory]{PicPhy}, \cite[Theorem 6.2.5]{Picard}}}]\label{t:st} Let $A$ be a skew-selfadjoint operator in $\mathcal{H}$, $0\leq M=M^*,N\in L(\mathcal{H})$. Assume there exists $c,\nu>0$ such that for all $\mu\geq\nu$, we have
\begin{equation}\label{e:st}
   \mu\langle  M \phi,\phi\rangle + \Re \langle N \phi,\phi \rangle \geq c\langle \phi,\phi\rangle\quad  (\phi\in \mathcal{H}).
\end{equation}
Then the operator $\mathcal{B}_\mu \coloneqq \partial_{t,\mu} \mathcal{M}+\mathcal{N}+\mathcal{A}$ with $D(\mathcal{B}_\mu)=D(\partial_{t,\mu})\cap D(\mathcal{A})$ is closable in $L_{\mu}^2(\mathbb{R};\mathcal{H})$. Moreover, $\mathcal{S}_\mu \coloneqq \overline{\mathcal{B}}_\mu^{-1}$ is well-defined, continuous and bounded with $\|\mathcal{S}_\mu\|_{L(L_\mu^2)}\leq 1/c$.  
\end{theorem}

\begin{remark}\label{r:in} In the situation of Theorem \ref{t:st}, assume there is $\eta\in \R$ with the property that $\overline{\mathcal{B}}_{\mu}^{-1}|_{C_c^\infty(\mathbb{R};\mathcal{H})}$ extends to a bounded linear operator $\mathcal{S}_\zeta\in L(L_\zeta^2(\mathbb{R};\mathcal{H}))$ for all $\zeta\geq\eta$. Then, by \cite[Theorem 6.1.4]{Picard} or \cite[Lemma 3.6]{Trostorff2013a}, for all $\zeta,\xi\geq \eta$ we have that $\mathcal{S}_{\xi} = \mathcal{S}_\zeta$ on $L_{\zeta}^2(\mathbb{R};\mathcal{H})\cap L_{\xi}^2(\mathbb{R};\mathcal{H})$. 
\end{remark}

\begin{remark}\label{r:reg} A consequence of Theorem \ref{t:st} is that $D(\mathcal{B}_\mu)=D(\partial_{t,\mu})\cap D(\mathcal{A})$ is an operator core for $\overline{\mathcal{B}}_\mu$. In particular, for $f\in L_\mu^2(\mathbb{R};\mathcal{H})$, there exists $(u_n)_n$ in $D(\mathcal{B}_\mu)$ converging in the graph norm of $\overline{\mathcal{B}}_\mu$ to $u$. Hence, 
\begin{align*}
   \partial_{t,\mu}^{-1} f & = \partial_{t,\mu}^{-1} \overline{\mathcal{B}}_\mu u  
    \\ & = \lim_{n\to\infty} \partial_{t,\mu}^{-1} \overline{\mathcal{B}}_\mu u_n 
    \\ & = \lim_{n\to\infty} \partial_{t,\mu}^{-1} (\partial_{t,\mu} \mathcal{M}+\mathcal{N} + \mathcal{A}) u_n
    \\ & = \lim_{n\to\infty}  (\partial_{t,\mu} \mathcal{M}+\mathcal{N} + \mathcal{A}) \partial_{t,\mu}^{-1} u_n,
\end{align*}
where we used Hille's Theorem to deduce that $\partial_{t,\mu}^{-1} \mathcal{A} \subseteq \mathcal{A}\partial_{t,\mu}^{-1}$. Further, we realize that the limit 
\[
   \lim_{n\to\infty}  (\partial_{t,\mu} \mathcal{M}+\mathcal{N}) \partial_{t,\mu}^{-1} u_n
\]
exists and equals $(\partial_{t,\mu} \mathcal{M}+\mathcal{N}) \partial_{t,\mu}^{-1} u$. Thus, by the closedness of $\mathcal{A}$ and continuity of $\partial_{t,\mu}^{-1}$, we infer $\lim_{n\to\infty} \partial_{t,\mu}^{-1}u_n=\partial_{t,\mu}^{-1}u\in D(\mathcal{A})$. Hence, $\partial_{t,\mu}^{-1} u\in D(\mathcal{B}_\mu)$ and
\[
  \partial_{t,\mu}^{-1}f = \partial_{t,\mu} \mathcal{M}\partial_{t,\mu}^{-1} u+\mathcal{N}\partial_{t,\mu}^{-1} u + \mathcal{A}\partial_{t,\mu}^{-1} u.
\]
Therefore, since $\partial_{t,\mu}^{-1}$ maps onto $D(\partial_{t,\mu})$, we get $\overline{\mathcal{B}}_\mu^{-1}[D(\partial_{t,\mu})]\subseteq D(\partial_{t,\mu})\cap D(\mathcal{A}).$
\end{remark}

Next, we define exponential stability in the present context.

\begin{definition}[{{\cite[Definition 3.1]{Trostorff2013a}}}] Let $A$ be a skew-selfadjoint operator in $\mathcal{H}$, $0\leq M=M^*,N\in L(\mathcal{H})$ satisfying \eqref{e:st}. Then $\mathcal{S}_\mu$ from Theorem \ref{t:st} is called \emph{exponentially stable}, if there exists $\eta>0$ such that 
\[
   \mathcal{S}_\mu f \in \bigcap_{-\eta < \zeta \leq \mu} L_{\zeta}^2(\mathbb{R};\mathcal{H})\quad\Big(f\in L_{-\eta}^2(\mathbb{R};\mathcal{H})\cap L_{\mu}^2(\mathbb{R};\mathcal{H})\Big).
\] 
\end{definition}

\begin{remark}
 We refer to \cite[Initial value problems]{Trostorff2013a}, for a relationship of the latter definition to the more commonly known notion of exponential stability for initial value problems. 
\end{remark}

We recall a criterion for exponential stability particularly interesting for the present situation. (Note that the strict positiveness of $M_0$ used in \cite[Theorem 4.1]{Trostorff2013a} is not needed.)

\begin{theorem}[{{\cite[Theorem 4.1]{Trostorff2013a}}}]\label{t:es} Let $A$ be skew-selfadjoint in $\mathcal{H}$, $0\leq M=M^*,N\in L(\mathcal{H})$. In addition to \eqref{e:st}, assume there exists $c'>0$ such that $2\Re N=N+N^*\geq c'$ in the sense of positive definiteness. Then $\mathcal{S}_\mu$ from Theorem \ref{t:st} is exponentially stable.
\end{theorem}

Next, we recall a result related to homogenization of the equations under consideration. The weak operator topology will be denoted by $\tau_{w}$. 

\begin{theorem}[{{\cite[Theorem 3.5]{Waurick2013AA_Hom} or \cite[Theorem 4.1]{Waurick2014SIAM_HomFrac}}}]\label{t:ht} Let $A=-A^*$ in $\mathcal{H}$ such that $D(A)$ (endowed with the graph norm) is compactly embedded into $\mathcal{H}$. Let $(M_n)_{n\in\mathbb{N}}$, $(N_n)_{n\in \mathbb{N}}$ be sequences in $L(\mathcal{H})$ with $0\leq M_n=M_n^*$, $n\in \mathbb{N}$. Assume there are $O,P\in L(\mathcal{H})$ such that $M_n\to O$ and $N_n\to P$ as $n\to\infty$ in $\tau_{w}$ of $L(\mathcal{H})$. Further, assume there is $c,\nu>0$ such that \eqref{e:st} holds for $M_n$ and $N_n$ instead of $M$ and $N$, respectively, for any $n\in \mathbb{N}$.

Then 
\[
   \mathcal{S}_{\mu,n}\coloneqq \overline{\partial_{t,\mu}\mathcal{M}_n+\mathcal{N}_n+\mathcal{A}}^{-1} \to  \overline{\partial_{t,\mu}\mathcal{O}+\mathcal{P}+\mathcal{A}}^{-1}
\]
in $\tau_w$ of $L(L_\mu^2(\mathbb{R};\mathcal{H}))$.
\end{theorem}

\begin{remark}
   Note that $\overline{\partial_{t,\mu}\mathcal{O}+\mathcal{P}+\mathcal{A}}^{-1}$ is a well-defined continuous linear operator in $L(\mathcal{H})$. Indeed, the condition \eqref{e:st} is stable under limits in the weak operator topology.
\end{remark}

\section{The Hyperbolic-Elliptic System} 

To begin with, we put the equation to study into a functional analytic perspective. Recall that for $n\in \mathbb{N}$ and a given $f$ we want to solve
\begin{equation}\label{e:2t}
   \begin{cases}
        \partial_t^2 u_n(t,x) - \partial_x^2 u_n(t,x) = \partial_t f(t,x),&  x\in \bigcup_{j\in \{1,\ldots,n\}} \big(\frac{(j-1)}{n},\frac{2j-1}{2n}\big)\\
        u_n(t,x) - \partial_x^2 u_n(t,x) = f(t,x),& x\in \bigcup_{j\in \{1,\ldots,n\}} \big(\frac{2j-1}{2n},\frac{j}{n}\big)\\
        \partial_x u_n(t,0)=\partial_x u_n(t,1)=0, & 
   \end{cases}
   \quad(t\in \mathbb{R}),
\end{equation} subject to homogeneous initial conditions and conditions for continuity at the junction points $\{(2j-1)/2n; j\in \{1,\ldots,n\}\}$ for $u_n$. 
On $\mathbb{R}\times \bigcup_{j\in \{1,\ldots,n\}} \big((j-1)/n,(2j-1)/2n\big)$ consider
\begin{equation}\label{e:o}
   \Big(\partial_t \begin{pmatrix}
                1 & 0 \\ 0 & 1 
              \end{pmatrix} - \begin{pmatrix} 0 & \partial_x \\ \partial_x & 0 \end{pmatrix}\Big)\begin{pmatrix} u_n \\ w_n
              \end{pmatrix} = \begin{pmatrix}  f  \\ 0 \end{pmatrix}
\end{equation}
as well as on $\mathbb{R}\times \bigcup_{j\in \{1,\ldots,n\}} \big((2j-1)/2n,j/n\big)$
\begin{equation}\label{e:e}
   \Big( \begin{pmatrix}
                1 & 0 \\ 0 & 1 
              \end{pmatrix} - \begin{pmatrix} 0 & \partial_x \\ \partial_x & 0 \end{pmatrix}\Big)\begin{pmatrix} u_n \\ w_n
              \end{pmatrix} = \begin{pmatrix} f  \\ 0 \end{pmatrix}.
\end{equation}
It is easy to check that formally the solution $u_n$ to both equations \eqref{e:o} and \eqref{e:e} lead to the first two equations of \eqref{e:2t}. Next, we write the two equations \eqref{e:o} and \eqref{e:e} within one single equation
\begin{equation}\label{e:1t}
 \Big(\partial_t \begin{pmatrix}
                1_n & 0 \\ 0 & 1_n 
              \end{pmatrix} + \begin{pmatrix}
                1-1_n & 0 \\ 0 & 1-1_n 
              \end{pmatrix} - \begin{pmatrix} 0 & \partial_x \\ \partial_x & 0 \end{pmatrix}\Big)\begin{pmatrix} u_n \\ w_n
              \end{pmatrix} = \begin{pmatrix}  f  \\ 0 \end{pmatrix},
\end{equation}
where $1_n$ denotes the multiplication operator induced by
\[
   a_n\colon x\mapsto a(nx),\quad \text{where } a\coloneqq \sum_{k\in \mathbb{Z}} \chi_{[k,k+1/2]} \in L^\infty(\mathbb{R}).
\]
($\chi_K$ denotes the characteristic function of a set $K$, that is, $\chi_K(x)=1$ if $x\in K$ and $\chi_K(x)=0$, if $x\notin K$.)
In order to account for the boundary conditions of $\partial_x u_n$ (see the third line in \eqref{e:2t}), we define
\[
   \partial_{x,0} : H_0^1(0,1)\subseteq L^2(0,1)\to L^2(0,1), u\mapsto u'.
\]
There are no boundary conditions for $u_n$. Hence, we let $\partial_x \coloneqq -\partial_{x,0}^*$ and the equation \eqref{e:1t} thus reads
   \begin{equation}\label{e:1tbc}
 \Big(\partial_t \begin{pmatrix}
                1_n & 0 \\ 0 & 1_n 
              \end{pmatrix} + \begin{pmatrix}
                1-1_n & 0 \\ 0 & 1-1_n 
              \end{pmatrix} - \begin{pmatrix} 0 & \partial_{x,0} \\ \partial_x  & 0 \end{pmatrix}\Big)\begin{pmatrix} u_n \\ w_n
              \end{pmatrix} = \begin{pmatrix}  f  \\ 0 \end{pmatrix}.
\end{equation}
We will address the conditions of continuity on the junction points after having shown well-posedness for \eqref{e:1tbc}. We apply Theorem \ref{t:st} with $\mathcal{H}=L^2(0,1)^2$ and
\begin{equation}\label{e:set}
   M_n \coloneqq \begin{pmatrix}
                1_n & 0 \\ 0 & 1_n 
              \end{pmatrix},\quad N_n \coloneqq \begin{pmatrix}
                1-1_n & 0 \\ 0 & 1-1_n 
              \end{pmatrix},\quad A = - \begin{pmatrix} 0 & \partial_{x,0} \\ \partial_x  & 0 \end{pmatrix}.
\end{equation}
It is easy to see that for all $\nu>0$ we have $\nu M_n+\Re N_n\geq \min\{1,\nu\}$ and that $A=-A^*$. Hence, we get the following theorem:
\begin{theorem}\label{t:sts} For all $\mu>0$, the operator 
\[
   \overline{\mathcal{B}}_{\mu,n} \coloneqq  \overline{\Big(\partial_{t,\mu} \begin{pmatrix}
                1_n & 0 \\ 0 & 1_n 
              \end{pmatrix} + \begin{pmatrix}
                1-1_n & 0 \\ 0 & 1-1_n 
              \end{pmatrix} - \begin{pmatrix} 0 & \partial_{x,0} \\ \partial_x  & 0 \end{pmatrix}\Big)}
\]
is continuously invertible in $L_{\mu}^2(\mathbb{R};L^2(0,1)^2)$. 
\end{theorem}

\begin{remark}\label{r:jc} By Remark \ref{r:reg}, we infer that if $f$ is weakly differentiable with respect to time, then so is $u_n$. Moreover, $(u_n,w_n)\in D(A)$. Hence, in particular, we obtain that $w_n\in D(\partial_{x,0})$ and $u_n\in D(\partial_{x})$. By Sobolev's embedding theorem, we deduce that both $u_n$ and $w_n$ are continuous with respect to the spatial variables. In particular, $u_n$ is $t$-almost everywhere continuous on the junction points.  
\end{remark}

In order to let $n\to\infty$ in \eqref{e:1tbc}, we recall the following well-known observation.
\begin{theorem}[{{see e.g.~\cite[Theorem 2.6]{CioDon}}}]\label{t:pm} Let $a\colon \mathbb{R}\to \mathbb{C}$ bounded, measurable and $1$-periodic. Then $ a(n\cdot)\to\int_{0}^1a(x)d x$ in the weak* topology of $L^\infty(\R)$ as $n\to\infty$. 
\end{theorem}

Next, for applying Theorem \ref{t:ht}, it, thus, suffices to observe that weak* convergence in $L^\infty(\mathbb{R})$ is the same as convergence of the associated multiplication operators in $L^2(\mathbb{R})$ and that both the spaces $D(\partial_x)=H^1(0,1)$ and $D(\partial_{x,0})=H_0^1(0,1)$ are compactly embedded into $L^2(0,1)$. Hence, $D(A)$ is compactly embedded in $L^2(0,1)^2$. Therefore, we obtain with the help of Theorem \ref{t:ht} applied with the settings as in \eqref{e:set}:
\begin{theorem}\label{t:hs} For every $\mu>0$ we have with  $\overline{\mathcal{B}}_{\mu,n}$ from Theorem \ref{t:sts} that
\[
   \overline{\mathcal{B}}_{\mu,n}^{-1} \eqqcolon \mathcal{S}_{\mu,n} \to \overline{\Big(\partial_{t,\mu} \begin{pmatrix}
                1/2 & 0 \\ 0 & 1/2 
              \end{pmatrix} + \begin{pmatrix}
                1/2 & 0 \\ 0 & 1/2 
              \end{pmatrix} - \begin{pmatrix} 0 & \partial_{x,0} \\ \partial_x  & 0 \end{pmatrix}\Big)}^{-1}\eqqcolon \mathcal{T}_\mu
\]in the weak operator topology of $L(L_{\mu}^2(\mathbb{R};L^2(0,1)^2))$.
\end{theorem}
\begin{proof}
 With the help of Theorem \ref{t:ht}, it suffices to observe that both $1_n$ and $1-1_n$ converge in the weak operator topology to $1/2$, by Theorem \ref{t:pm}.
\end{proof}

Next, it is an application of Theorem \ref{t:es} that $\mathcal{T}_\mu$ is exponentially stable:

\begin{theorem}\label{t:ess} For every $\mu>0$ we have that $\mathcal{T}_\mu$ from Theorem \ref{t:hs} is exponentially stable.
\end{theorem}
\begin{proof}
 The assertion follows by observing that $\Re \begin{pmatrix}
                1/2 & 0 \\ 0 & 1/2 
              \end{pmatrix}\geq 1/2>0$ and by applying Theorem \ref{t:es}.
\end{proof}

We conclude with observing that for $(u,w)$ with the property $\mathcal{T}_\mu (f,0) = (u,w)$ we obtain for smooth $f$ (see also Remark \ref{r:reg})
\[
   \Big(\partial_{t} \begin{pmatrix}
                1/2 & 0 \\ 0 & 1/2 
              \end{pmatrix} + \begin{pmatrix}
                1/2 & 0 \\ 0 & 1/2 
              \end{pmatrix} - \begin{pmatrix} 0 & \partial_{x,0} \\ \partial_x  & 0 \end{pmatrix}\Big) \begin{pmatrix} u \\ w\end{pmatrix} = \begin{pmatrix} f \\ 0\end{pmatrix}.
\]
Reading off the second line, we get
\[
  \partial_t w + w - 2\partial_{x} u = 0\text{ or } w = (1+\partial_t)^{-1}2\partial_{x} u.
\]
Thus, the first line reads
\[
   (\partial_t  + 1)u - 2\partial_{x,0} w = (\partial_t  + 1)u - 2\partial_{x,0} (1+\partial_t)^{-1}2\partial_{x} u= 2f.
\]
So,
\[
   \partial_t^2 u + 2\partial_t u + u - 4\partial_{x,0}\partial_{x} u= 2f+2\partial_t f,
\]
which establishes \eqref{e:hh}.

\section{The Hyperbolic-Parabolic System}

In the concluding section, we consider
\begin{equation}\label{e:2p}
   \begin{cases}
        \partial_t^2 u_n(t,x) - \partial_x^2 u_n(t,x) = \partial_t f(t,x),&  x\in \bigcup_{j\in \{1,\ldots,n\}} \big(\frac{j-1}{n},\frac{2j-1}{2n}\big)\\
        \partial_t u_n(t,x) - \partial_x^2 u_n(t,x) = f(t,x),& x\in \bigcup_{j\in \{1,\ldots,n\}} \big(\frac{2j-1}{2n},\frac{j}{n}\big)\\
        \partial_x u_n(t,0)=\partial_x u_n(t,1)=0, & 
   \end{cases}
   \quad(t\in \mathbb{R}),
\end{equation} subject to homogeneous initial conditions and conditions for continuity at the junction points $\{(2j-1)/2n; j\in \{1,\ldots,n\}\}$ for $u_n$. 
As the arguments are similar (if not entirely the same) to the case treated in the previous section, we will not give the details here. 
Rewritten as a $(2\times2)$-block operator matrix system, equation \eqref{e:2p} reads
\begin{equation}\label{e:22p}
   \Big(\partial_t \begin{pmatrix}
                1 & 0 \\ 0 & 1_n 
              \end{pmatrix}+\begin{pmatrix}
                0 & 0 \\ 0 & 1-1_n 
              \end{pmatrix} - \begin{pmatrix} 0 & \partial_{x,0} \\ \partial_x & 0 \end{pmatrix}\Big)\begin{pmatrix} u_n \\ w_n
              \end{pmatrix} = \begin{pmatrix}  f  \\ 0 \end{pmatrix}
\end{equation} as an equation on $\mathbb{R}\times (0,1)$. 
So, $((u_n,w_n))_n$ converges weakly to the solution $(u,w)$ of
\begin{equation}\label{e:22h}
   \Big(\partial_t \begin{pmatrix}
                1 & 0 \\ 0 & 1/2 
              \end{pmatrix}+\begin{pmatrix}
                0 & 0 \\ 0 & 1/2 
              \end{pmatrix} - \begin{pmatrix} 0 & \partial_{x,0} \\ \partial_x & 0 \end{pmatrix}\Big)\begin{pmatrix} u \\ w
              \end{pmatrix} = \begin{pmatrix}  f  \\ 0 \end{pmatrix}
\end{equation}
Thus, written as a second order system, we get the following equation for $u$:
\[
  \partial_t^2 u + \partial_t u - 2\partial_{x,0}\partial_x u = f +\partial_t f.
\]
Let $f(t,x)=\phi(t)\chi_{[0,1]}(x)$ for some non-negative, compactly supported, smooth function $\phi\neq 0$. Then, with the ansatz $u(t,x)=\psi(t)\chi_{[0,1]}(x)$, we arrive at
\[
  \partial_t^2 \psi + \partial_t \psi = \phi +\partial_t \phi.
\]
So,
\[
  \partial_t \psi = \phi.
\]
Thus, as $\phi$ is positive and compactly supported, $\psi(t)=\int_{-\infty}^t \phi(\tau)d\tau$ is eventually constant. Hence, the limit equation is \emph{not} exponentially stable in the sense of \cite[Definition 3.1]{Trostorff2013a} (see also \cite[Section 3.1]{Trostorff2015a}).

\noindent
Marcus Waurick \\
University of Bath \\
Department of Mathematical Sciences \\
BA2 7AY Bath, England, UK \\
{\tt m.wau\rlap{\textcolor{white}{hugo@egon}}rick@bath\rlap{\textcolor{white}{darmstadt}}.ac.uk}

\end{document}